\documentclass[12pt,notitlepage]{amsart}
\usepackage{latexsym,amsfonts,amssymb,amsmath,amsthm} %{pstricks,pst-node,epsf}
\newcommand{\mz}{\ensuremath{\mathbb Z}}

\newcommand{\mc}{\ensuremath{\mathbb C}}

\newcommand{\shortmod}{\ensuremath{\negthickspace \negthickspace \negthickspace \pmod}}

\newcommand{\half}{\ensuremath{ \frac{1}{2}}}

\newcommand{\sumstar}{\sideset{}{^*}\sum}
\newcommand{\sumplus}{\sideset{}{^+}\sum}

\newcommand{\sign}{\mathfrak{a}}
\newcommand{\chibar}{\overline{\chi}}

\theoremstyle{plain}		
	\newtheorem{mytheo}{Theorem}[section]

	\newtheorem{myprop}[mytheo]{Proposition}
	\newtheorem{mycoro}[mytheo]{Corollary}
     \newtheorem{mylemma}[mytheo]{Lemma}

\theoremstyle{remark}

\begin{document}
\today
\title{The reciprocity law for the twisted second moment of Dirichlet L-functions}
\author{Matthew P. Young} 
\address{Department of Mathematics,
Texas A\&M University,
College Station, TX 77843-3368}
\email{myoung@aimath.org}
\thanks{This research was supported by an NSF Mathematical Sciences Post-Doctoral Fellowship and by the American Institute of Mathematics.}
\begin{abstract}
We produce a new proof of the reciprocity law for the twisted second moment of Dirichlet L-functions that was recently proved by Conrey.  Our method is to analyze certain two-variable sums where the variables satisfy a linear congruence.  We show that these sums satisfy an elegant reciprocity formula.  In the case that the modulus is prime, these sums are closely related to the twisted second moment, and the reciprocity formula for these sums implies Conrey's reciprocity formula.  We also extend the range of uniformity of Conrey's formula.
\end{abstract}
\maketitle
\section{Introduction}
For coprime integers $h$ and $p>0$, consider
\begin{equation}
M(p,h) = \sumstar_{\chi \shortmod{p}} |L({\textstyle \half}, \chi)|^2 \chi(h),
\end{equation}
where the $*$ indicates the summation is over all primitive characters (we use similar notation with $+$ or $-$ to denote summation over all primitive, even or odd characters).

In a recent paper, Conrey \cite{Conrey} proved a kind of reciprocity formula relating $M(p,h)$ and $M(h,-p)$ (with both $p$ and $h$ prime), and remarked that such formulas deserve further study.  In this paper, we give a different, more direct, proof of the reciprocity formula that incidentally extends its range of uniformity.  Conrey showed (see Theorem 10 of \cite{Conrey}), for certain explicit constants $A$ and $B$, that
\begin{equation}
M(p,h) = \frac{\sqrt{p}}{\sqrt{h}} M(h,-p) + \frac{p}{\sqrt{h}}\left(\log \frac{p}{h} + A \right) + \frac{B}{2} \sqrt{p} + O(h + \log{p} + \sqrt{p/h} \log{p}),
\end{equation} 
which provides an asymptotic formula for $M(p,h) - \sqrt{p/h} M(h,-p)$ provided $h < p^{2/3}$.  In this paper we improve the error term above so that the asymptotic formula holds for $h < p^{1 - \varepsilon}$.

For prime $p$, the sum $M(p,h)$ is related to sums of the type
\begin{equation}
S(p,a;f) = \mathop{\sum \sum}_{n \equiv am \shortmod{p}} \frac{1}{\sqrt{mn}} f\left(\frac{mn}{p}\right),
\end{equation}
where $f$ is a nice function, and $a = \pm h$.  It turns out that $S(p,h;f)$ itself is related to $S(h,-p;f)$.  Conrey's reciprocity formula is then a consequence of this relation.  It is only for prime $p$ and $h$ that there is such a nice relation between the twisted second moments, but the reciprocity relation for $S(p,h;f)$ and $S(h,-p;f)$ holds for non-prime $h$ and $p$.  Precisely, we have
\begin{mytheo}
\label{thm:mainresult}
Suppose $a$ and $q$ are positive coprime integers and that $f$ is a smooth function with Mellin transform $\widetilde{f}(s)$ that is meromorphic for $s \in \mc$ with a possible pole at $s=0$ only, with rapid decay as $|s| \rightarrow \infty$ in any fixed strip $-B \leq \text{Re}(s) \leq B$.  Then
\begin{equation}
\label{eq:mainresultplus}
S(q,a;f) = a^{-\half}(c_1 \log(q/a) + c_2) + c_3 q^{-\half} +  \sqrt{\frac{a}{q}} S(a,-q;f) +O(q^{-\half + \varepsilon} \frac{a}{q}),
\end{equation}
where $c_1$, $c_2$, and $c_3$ are certain constants (depending on $f$ only).  Similarly,
\begin{equation}
\label{eq:mainresultminus}
S(q,-a;f) = 
\sqrt{\frac{a}{q}} S(a,q;f) + c_3 q^{-\half}+ O\left(q^{-\half+ \varepsilon} \frac{a}{q}\right).
\end{equation}
\end{mytheo}
Remarks.  
\begin{itemize}
\item This theorem gives an asymptotic formula for $S(q,\pm a;f) - \sqrt{\frac{a}{q}} S(a, \mp q;f)$ provided $a < q^{1 - \varepsilon}$.
\item The reciprocity relation is not self-dual, so it could potentially be used recursively to obtain a curious kind of asymptotic expansion of $S(q,a;f)$.
\item The conditions on $f$ are natural to require for the purpose of proving Corollary \ref{coro:reciprocity} and can certainly be loosened.
\end{itemize}
\begin{mycoro}
\label{coro:reciprocity}
Suppose $h < p$ are primes.  Then
\begin{equation}
M(p,h) = \frac{p}{\sqrt{h}} (\log(p/h)+A) + \frac{ \sqrt{p}}{ \sqrt{h}} M(h,-p) 
+\zeta({\textstyle \half})^2 \sqrt{p}  + 
+O(p^{-\half + \varepsilon} h + h^{-\half+\varepsilon} p^{\half} ),
\end{equation}
where $A = \gamma - \log(8 \pi)$.
\end{mycoro}
Remark.  We actually obtained a more precise formula with an error term of size $p^{\half} h^{-C} + p^{-\half + \varepsilon} h$ where $C > 0$ is arbitrary.  See \eqref{eq:moreprecise} below.

Selberg showed \cite{Selberg} (for $\alpha, \beta \ll (\log{q})^{-1}$, $\alpha + \beta \neq 0$, say)
\begin{multline}
\label{eq:Selberg}
\sumstar_{\chi \shortmod{q}} L({\textstyle \half }+ \alpha, \chi) L({\textstyle \half } + \beta, \chibar) \chi(h) \chibar(k) = \frac{\phi(q)}{h^{\half + \alpha} k^{\half + \beta}}  \zeta_q(1+\alpha + \beta) 
+\frac{\phi^*(q) }{h^{\half - \alpha} k^{\half - \beta}} Y_{\alpha,\beta} \zeta_q(1-\alpha - \beta) 
\\
+ O((h q^{\half} + k q^{\half} + hk q^{\frac13})q^{\varepsilon}),
\end{multline}
where
\begin{equation}
Y_{\alpha,\beta} =  \left(\frac{q}{2\pi}\right)^{-\alpha -\beta} \pi^{-1} \Gamma({\textstyle \half} -\alpha) \Gamma({\textstyle \half}  - \beta)\cos({\textstyle \frac{\pi}{2}}(\alpha - \beta),
\end{equation}
$q$ is prime, and $(h,k) = (hk,q) = 1$.  This gives an asymptotic formula provided $hk < q^{\frac49 - \varepsilon}$ and $hk \max(h^2, k^2) < q^{1-\varepsilon}$.  In particular, if $k=1$, then $h$ can be as large as $q^{\frac13 - \varepsilon}.$

Iwaniec and Sarnak \cite{IS} showed
\begin{equation}
\sumplus_{\chi \shortmod{q}} L({\textstyle \half }, \chi) L({\textstyle \half }, \chibar) \chi(h) \chibar(k) = \frac{\phi^+(q) \phi(q)}{q \sqrt{hk}} \log(L^2(q)/hk)+ \beta(h,k),
\end{equation}
where $L(q)$ is defined by
\begin{equation}
\log{L(q)} = \log{\sqrt{q/\pi}} + {\textstyle \half} \psi({\textstyle \frac14 }) + \gamma + \sum_{p | q} \frac{\log{p}}{p-1},
\end{equation}
and
\begin{equation}
\beta(h,k) = \mathop{\sum \sum}_{hm \neq kn} \frac{(hm \pm kn, q)}{\sqrt{mn}} |W(mn/q)|,
\end{equation}
where $W$ is a smooth, bounded function of rapid decay.  They also gave an average bound for $\beta(h,k)$:  
\begin{equation}
\sum_{m,n \leq M} \frac{\beta(m,n)}{\sqrt{mn}} \ll d(q) \sqrt{q} M (\log{qm})^4.
\end{equation}
Their application was to prove that $1/3$ of Dirichlet $L$-functions do not vanish (nor are too small) at the central point.

I thank Brian Conrey for showing me preliminary versions of \cite{Conrey} and for various conversations about his work.

\section{Setup}
Suppose $\chi$ is a primitive Dirichlet character of conductor $q$.  The Dirichlet $L$-function is defined by
\begin{equation}
L(s,\chi) = \sum_n \frac{\chi(n)}{n^s}, \quad \text{Re}(s) > 1.
\end{equation}
The completed L-function  
\begin{equation}
\Lambda(s, \chi) = \left(\frac{q}{\pi}\right)^{\frac{s}{2}} \Gamma\left(\frac{s + \sign}{2} \right) L(s, \chi)
\end{equation}
satisfies the functional equation
\begin{equation}
\Lambda(s, \chi) = \epsilon(\chi) \Lambda(1-s, \chibar),
\end{equation}
where
\begin{equation}
\epsilon(\chi) = i^{-\sign} q^{-\half} \tau(\chi),
%\end{equation}
%\begin{equation}
\qquad
\sign = 
\begin{cases}
0, & \chi(-1) = 1 \\
1, & \chi(-1) = -1,
\end{cases}
\end{equation}
and $\tau(\chi)$ is the Gauss sum
\begin{equation}
\tau(\chi) = \sum_{x \shortmod{q}} \chi(x) e\left(\frac{x}{q}\right).
\end{equation}
% In its asymmetric form, the functional equation reads
% \begin{equation}
% L(s, \chi) = \epsilon(\chi) X(s) L(1-s, \chibar),
% \end{equation}
% where
% \begin{equation}
% \label{eq:Xfactor}
% X(\half + u) = \left(\frac{q}{\pi} \right)^{-u } \frac{\Gamma\left(\frac{\half - u + \sign}{2} \right)}{\Gamma\left(\frac{\half + u + \sign}{2} \right)}.
% \end{equation}

We need
\begin{mylemma}[Approximate functional equation]  Let $G(s)$ be an entire function satisfying the decay condition $G(s) \ll_{A,B} (1 + |s|)^{-A}$ for any $A > 0$ in any strip $-B \leq \text{Re}(s) \leq B$.  Furthermore assume $G(-s) = G(s)$ and that $G$ has a double zero at each $s \in \half + \mz$.  Then
\begin{equation}
L({\textstyle \half}, \chi) L({\textstyle \half},\chibar) = 2 \sum_m \sum_n \frac{\chi(m) \chibar(n)}{\sqrt{mn}} V_{\sign}\left(\frac{mn}{q}\right),
\end{equation}
where
\begin{equation}
V_{\sign}(x) = \frac{1}{2\pi i} \int_{(1)} \pi^{-s} \frac{G(s)}{s} \frac{\Gamma\left(\frac{\half + s + \sign}{2}\right)^2}{\Gamma\left(\frac{\half + \sign}{2}\right)^2} x^{-s} ds.
\end{equation}
\end{mylemma}
For a proof, see \cite{Sound}, Lemma 2.  The point of having $G$ vanish at half integers is to cancel the poles of the gamma functions.  As an example of an allowable function $G$, take $G(s) = e^{s^2} \cos^2(\pi s)$.

To average over primitive characters, we require
\begin{mylemma} 
\label{prop:orthogonality}
If $(ab, q) =1$ then
\begin{equation}
\sideset{}{^+}\sum_{\chi \shortmod{q}} \chi(a)\chibar(b) = \half \sum_{\substack{d | q \\ d| a - b}} \phi(d) \mu(q/d) + \half \sum_{\substack{d | q \\ d| a + b}} \phi(d) \mu(q/d).
\end{equation}
The sum on the left hand side vanishes if $(ab, q) \neq 1$. Furthermore,
\begin{equation}
\sideset{}{^-}\sum_{\chi \shortmod{q}} \chi(a)\chibar(b) = \half \sum_{\substack{d | q \\ d| a - b}} \phi(d) \mu(q/d) - \half \sum_{\substack{d | q \\ d| a + b}} \phi(d) \mu(q/d).
\end{equation}
\end{mylemma}

Using the approximate functional equation and Lemma \ref{prop:orthogonality}, we obtain
\begin{mylemma} 
\label{lemma:avgAFE}
Suppose $h$ and $p>0$ are coprime integers.  Then we have
\begin{equation}
M(p,h) = 2\sum_{d | p} \phi(d) \mu(p/d) \left(\mathop{\sum \sum}_{\substack{(mn,p) = 1 \\ n \equiv hm \shortmod{d}}} \frac{1}{\sqrt{mn}} V_{+}\left(\frac{mn}{p}\right) +  \mathop{\sum \sum}_{\substack{(mn,p) = 1 \\  n \equiv -hm \shortmod{d}}} \frac{1}{\sqrt{mn}} V_{-}\left(\frac{mn}{p}\right)\right),
\end{equation}
where $V_{+} = \half(V_0 + V_1)$, and $V_{-} = \half(V_0 - V_1)$.
\end{mylemma}
In case $p$ is prime then we can simplify the above expression.
\begin{mylemma}
\label{lemma:summary}
Suppose $p$ is prime and $(h,p) = 1$.  Then
\begin{equation}
M(p,h) =  2\phi(p) S(p,h;V_{+}) +  2\phi(p) S(p,-h;V_{-}) -2 \zeta({\textstyle \half})^2(1 - p^{-\half}) + O_C(p^{-C}),
\end{equation}
where
\begin{equation}
\label{eq:Spm}
S(p,\pm h;V_{\pm}) = \mathop{\sum \sum}_{n \equiv \pm h m \shortmod{p}} \frac{1}{\sqrt{mn}} V_{\pm}\left(\frac{mn}{p}\right)
\end{equation}
and $C > 0$ is arbitrary.
\end{mylemma}

\begin{proof}
Using Lemma \ref{lemma:avgAFE}, we write
\begin{equation}
M(p,h) = 2\phi(p) S'(p,h;V_{+}) +  2\phi(p) S'(p,-h;V_{-}) -2S'(1,1;V_0,p),
\end{equation}
where
\begin{equation}
S'(p,\pm h;V_{\pm}) =  \mathop{\sum \sum}_{\substack{(mn,p) = 1 \\ n \equiv \pm hm \shortmod{p}}} \frac{1}{\sqrt{mn}} V_{\pm}\left(\frac{mn}{p}\right),
\end{equation}
and
\begin{equation}
S'(1,1;V_0,p) =   \mathop{\sum \sum}_{\substack{(mn,p) = 1}} \frac{1}{\sqrt{mn}} V_0\left(\frac{mn}{p}\right).
\end{equation}
First notice that $S'(p,\pm h;V_{\pm}) = S(p,\pm h;V_{\pm}) + O_{C}(p^{-C})$, since $p$ is prime and $V_{\pm}$ has rapid decay. 

The assumption that $G$ has a double zero at $s= \half$ implies that $S'(1,1;V_0,p)$ is small.  Precisely, we have
\begin{equation}
S'(1,1;V_0,p) =  \frac{1}{2\pi i} \int_{(1)} \frac{G(s)}{s}
\frac{\Gamma\left(\frac{\half + s}{2}\right)^2}{\Gamma\left(\frac{1}{4}\right)^2}\left(\frac{p}{\pi}\right)^{s} \zeta_p({\textstyle \half} + s)^2 ds. %= -2 \zeta_p({\textstyle \half})^2 + O(p^{-2007}),
\end{equation}
Write $\zeta_p({\textstyle \half} + s)^2 = \zeta({\textstyle \half} + s)^2 (1 - 2 p^{-\half - s} + p^{-1 - 2s})$ and correspondingly write $S'(1,1;V_p,p)$ as the sum of three integrals.  The first term is easily seen to be $\zeta({\textstyle \half})^2 + O(p^{-C})$ by moving $s$ to the line $-C$.  The third term is bounded by $\ll p^{-C}$ by moving $s$ to the right.  The second term is
\begin{equation}
-2 p^{-\half} \frac{1}{2\pi i} \int_{(1)} \pi^{-s} \frac{G(s)}{s}
\frac{\Gamma\left(\frac{\half + s}{2}\right)^2}{\Gamma\left(\frac{1}{4}\right)^2} \zeta({\textstyle \half} + s)^2 ds.
\end{equation}
It turns out that we can compute this term exactly because the integral is odd under $s \rightarrow -s$, which can be seen by applying the functional equation
\begin{equation}
\Gamma\left(\frac{\half + s}{2}\right) \zeta({\textstyle \half} + s) = \pi^{s} \Gamma\left(\frac{\half - s}{2}\right) \zeta({\textstyle \half} - s).
\end{equation}
Thus the integral is half the residue at $s=0$, that is
\begin{equation}
-p^{-\half} \zeta({\textstyle \half})^2.
\end{equation}
Thus
\begin{equation}
S'(1,1;V_0,p) = \zeta({\textstyle \half})^2 (1 - p^{-\half}) + O(p^{-2007}). \qedhere
\end{equation}
\end{proof}

The question presents itself to further analyze the sums $S(q,a,x)$ defined by 
\begin{equation}
S(q,a,x) = \sum_{\substack{mn \leq x \\ n \equiv am \shortmod{q}}} 1,
\end{equation}
where $(a, q) = 1$.  To some extent the sum $S(p,h,p)$ models $S(p,h;V_{+})$, but better results can be obtained for the smoothed sums.

\section{Upper bounds for $S(q,a,x)$}
In this section we show a method for obtaining upper bounds for the sum $S(q,a,x)$.
As a first estimate, we have the following
\begin{mylemma}
\label{lemma:simple}
Suppose $0 < a < q$ and $(a,q) = 1$.  Then for any $x \geq 1$, 
\begin{equation}
\label{eq:sudoku}
S(q,a,x) \leq 2 \sqrt{x} + 2 \frac{x}{q}  \log 3x.
\end{equation}
More generally, let 
\begin{equation}
S_{M,N}(q,a) = \sum_{\substack{M < m \leq 2M \\ N < n \leq 2N \\ n \equiv am \shortmod{q}}} 1.
\end{equation}
Then
\begin{equation}
S_{M,N}(q,a) \leq \min(M,N) + \frac{MN}{q}.
\end{equation}
\end{mylemma}
\begin{proof}
By Dirichlet's trick,
\begin{equation}
\label{eq:dirichlettrick}
S(q,a,x) \leq \sum_{m \leq \sqrt{x}} \sum_{\substack{n \leq \frac{x}{m} \\ n \equiv am \shortmod{q} }} 1 + \sum_{m \leq \sqrt{x}} \sum_{\substack{n \leq \frac{x}{m} \\ n \equiv \overline{a}m \shortmod{q} }} 1.
\end{equation}
By breaking up the sum over $n$ into arithmetic progressions $\pmod{q}$, we have for any $x, y, r, q$
\begin{equation}
\sum_{\substack{ x < n \leq x + y\\ n \equiv r \shortmod{q}}} 1 \leq 1 + \frac{y}{q},
\end{equation}
which upon insertion into \eqref{eq:dirichlettrick} gives \eqref{eq:sudoku}.  The estimate for $S_{M,N}(q,a)$ is similar.
\end{proof}
Remark.  This result is essentially best-possible without using more information about $a$ and $q$, since if $a=1$ then there is a main term of size $\sqrt{x}$, and of course $\frac{x}{q} \log{x}$ is an obvious expected order of magnitude (being the number of terms times the probability of satisfying a congruence $\pmod{q}$.)

Lemma \ref{lemma:simple} can be improved by further analysis.
\begin{myprop}
\label{thm:reciprocityUB}
Suppose $(a,q) = 1$ and $0 < a < q$.  Then
\begin{equation}
S(q,a,x) \leq  \frac{x}{q} \log{3x}  + \sqrt{\frac{x}{a}} + S(a,-q,\frac{ax}{q}).
\end{equation}
Consequently,
\begin{equation}
\label{eq:waterman}
S(q,a,x) \leq 3 \frac{x}{q} \log{3x} + \sqrt{\frac{x}{a}} + 2 \sqrt{\frac{x}{q/a}}).
\end{equation}
\end{myprop}
\begin{proof}
We split the sum $S(q,a,x)$ into three pieces depending on if $n = am$, $n > am$, or $n < am$.  We have
\begin{equation}
\sum_{n = am} = \sum_{am^2 \leq x} 1 \leq \sqrt{\frac{x}{a}}.
\end{equation}
For $n > am$ we write $n = am + ql$ with $l > 0$, obtaining
\begin{equation}
\sum_{n > am} = \sum_{m(am + ql) \leq x} 1 \leq \sum_{qlm \leq x} 1 \leq \frac{x}{q} \log{3x}.
\end{equation}
% \begin{equation}
% \sum_{n > am} = \sum_{m(am + ql) \leq x} 1 = \sum_{ql \leq \frac{x}{m} - am} 1 = \sum_{\frac{x}{m} - am \geq q} \left( \frac{\frac{x}{m} - am}{q} + O(1) \right).
% \end{equation}
% Note that $\frac{x}{m} - am \geq q$ is equivalent to $am^2 + qm \leq x$, which in turn is equivalent to $(m + \frac{q}{2a})^2 - \frac{q^2}{4a^2} \leq \frac{x}{a}$, that is $m \leq \sqrt{\frac{x}{a} + \frac{q^2}{4a^2}} - \frac{q}{2a}$. Note that
% \begin{equation}
% \frac{q}{2a} (\sqrt{1 + \frac{4xa}{q^2}} - 1) = \frac{q}{2a} \left(\frac{2xa}{q^2} + O(\frac{x^2 a^2}{q^4}) \right) = \frac{x}{q} + O(\frac{x^2 a}{q^3}).
% \end{equation}
For $n < am$ we write $n = am - ql$, with $l > 0$.  Then we obtain
\begin{equation}
\sum_{n < am} = \mathop{\sum \sum}_{\substack{n \equiv -ql \shortmod{a} \\ n(n+ ql) \leq ax}} 1 \leq \mathop{\sum \sum}_{\substack{n \equiv -ql \shortmod{a} \\ nl \leq \frac{ax}{q}}} 1 = S(a,-q,\frac{ax}{q}).
\end{equation}
Estimate \eqref{eq:waterman} follows by applying Lemma \ref{lemma:simple} to $S(a,-q,\frac{ax}{q})$.
\end{proof}
\begin{mycoro}
Suppose $a < x^{-1} (\log{3x})^{-2} q^2$ and $a < \sqrt{q}$.  Then
\begin{equation}
S(q,a,x) = \sqrt{\frac{x}{a}} + O\left(\frac{x}{q} \log{3x} + \sqrt{\frac{x}{q/a}} +1 \right).
\end{equation}
In particular, if $x \asymp q$, then
\begin{equation}
S(a,q,x) = \sqrt{\frac{x}{a}} + O(\sqrt{a} + \log{3x}).
\end{equation}
\end{mycoro}
The proof follows by simply noting that the diagonal terms (with $n = am$) contribute $\geq \sqrt{\frac{x}{a}} + O(1)$.

\section{Asymptotics for $S(q,a,f)$}
\label{section:asymptoticplus}
In this section we show how to develop an asymptotic expansion for a smoothed sum analogous to $S(q,a,x)$.  To that end, define
\begin{equation}
S(q,a;f,X) = \sum_{n \equiv am \shortmod{q}} \frac{1}{\sqrt{mn}} f\left(\frac{mn}{X}\right),
\end{equation}
where $f$ is a function satisfying the conditions of Theorem \ref{thm:mainresult} and $X \geq 1$ is a parameter (for the application of proving Theorem \ref{thm:mainresult} we shall take $X = q$). Write
\begin{equation}
f(x) = \frac{1}{2 \pi i} \int_{(c)} \widetilde{f}(s) x^{-s} ds,
\end{equation}
where
\begin{equation}
\widetilde{f}(s) = \int_0^{\infty} f(x) x^{s} \frac{dx}{x}.
\end{equation}
In the case that $f(x) = V_{\pm}(x)$ (that is the case for our primary application), we have
\begin{equation}
\label{eq:fV}
\widetilde{f}(s) = \pi^{-s} \frac{G(s)}{s} g_{\pm}(s),
\end{equation}
where
\begin{equation}
g_{\pm}(s) = \half\left(\frac{\Gamma\left(\frac{\half + s}{2}\right)^2}{\Gamma\left(\frac{1}{4}\right)^2} \pm \frac{\Gamma\left(\frac{\frac32 + s }{2}\right)^2}{\Gamma\left(\frac{3}{4}\right)^2} \right).
\end{equation}
Note $g_{-}(0) = 0$ and $g_{+}(0) = 1$, so that $\widetilde{f}(s)$ is meromorphic on $\mc$ with a pole at $s=0$ only in the $+$ case and is entire in the $-$ case.

The result is
\begin{mytheo}
\label{thm:reciprocityAF}
Suppose $0 < a < q$, $(a,q) = 1$, and $f$ is a function satisfying the conditions of Theorem \ref{thm:mainresult}.  Then
\begin{multline}
S(q,a;f,X) = \sqrt{\frac{a}{q}} S(a,-q;f, \frac{aX}{q}) + a^{-\half}\left(\frac{a_{-1}(f)}{2} \log(X/a) + \frac{a_0(f)}{2} + \gamma a_{-1}(f)\right) 
\\
+ q^{-\half}k\left(\frac{q}{X}\right)    
+ O\left(a^{-\half} (X/a)^{-A}
+ q^{-\half} \frac{a}{q} \left(\frac{X}{q}\right)^{\half} \left(1 + \frac{X}{q} \right) (qX)^{\varepsilon}
\right),
\end{multline}
where the $a_i(f)$ are given below by \eqref{eq:fcoeff} and $k$ is the function defined by
\begin{equation}
\label{eq:kdef}
k(y) = \frac{1}{2\pi i} \int_{(2)} y^{-s} \widetilde{f}(s) \zeta^2({\textstyle \half} + s) ds.
\end{equation}
The implied constant depends on $f$ and $\varepsilon$ only.
\end{mytheo}
Taking $X=q$ gives \eqref{eq:mainresultplus}.
%For the purposes of estimating $S(q,h;V_{+})$ we have $X=q$ and $a_{-1}(f) = 1$, in which case the theorem reads
% \begin{mycoro}  We have
% \label{coro:reciprocityAF}
% \begin{equation}
% S(q,a;V_{+}) = a^{-\half}({\textstyle \half} \log(q/a) + c_1) + c_2 q^{-\half} +  \sqrt{\frac{a}{q}} S(a,-q;V_{+}) +O(q^{-\half + \varepsilon} \frac{a}{q}),
% \end{equation}
% for certain constants $c_1$ and $c_2$.
% \end{mycoro}
% Remark.  Notice that Corollary \ref{coro:reciprocityAF} gives an asymptotic formula for $S(q,a;V_{+}) - \sqrt{\frac{a}{q}} S(a,-q;V_{+})$ provided $a < q^{1 - \varepsilon}$.

\begin{proof}[Proof of Theorem \ref{thm:reciprocityAF}]
As in the analysis of $S(q,a,x)$, write $S(q,a;f,X) = \sum_{n = am} + \sum_{n > am} + \sum_{n < am}$.  We have
\begin{equation}
\sum_{n=am} = \sum_m \frac{1}{m\sqrt{a}} f\left(\frac{am^2}{X}\right) = a^{-\half} \frac{1}{2\pi i} \int_{(c)} (X/a)^s \widetilde{f}(s) \zeta(1+2s) ds,
\end{equation}
for any $c > 0$.  Moving the line of integration to $-A$ (with $A >0$), we obtain a main term from the pole at $s=0$, and bounding the new integral trivially gives an error term of size
\begin{equation}
a^{-\half} (X/a)^{-A}.
\end{equation}
To compute the main term, write
\begin{align}
(X/a)^s & = 1 + s\log(X/a) + \dots \\
\zeta(1 + 2s) & = \frac{1}{2s} + \gamma + \dots \\
\label{eq:fcoeff}
\widetilde{f}(s) & = \frac{a_{-1}(f)}{s} + a_0(f) + \dots.
\end{align}
Thus we obtain
\begin{equation}
\label{eq:diag}
\sum_{n=am} = a^{-\half} \left( \frac{a_{-1}(f)}{2} \log(X/a) + \frac{a_0(f)}{2} + \gamma a_{-1}(f)\right) + O(a^{-\half} (X/a)^{-A}).
\end{equation}

For $n > am$, we have
\begin{align}
\sum_{n > am} &= \sum_{m, l > 0} \frac{1}{\sqrt{m(am+ql)}} f\left(\frac{m(am + ql)}{X}\right) 
\\
&= q^{-\half} \sum_{m, l} \frac{1}{\sqrt{ml}} \frac{1}{2 \pi i} \int_{(c)} \left(\frac{X}{q}\right)^s \widetilde{f}(s) (ml)^{-s} \left(1 + \frac{am}{ql}\right)^{-\half-s} ds.
\end{align}
Using the simple approximation
\begin{equation}
\label{eq:MVI}
(1 + x)^{-\half - s} = 1 + O(x |{\textstyle \half} + s|)
%(1 + \frac{am}{ql})^{-\half - s} = 1 + O(\frac{am}{ql} |{\textstyle \half} + s|),
\end{equation}
(valid for $\text{Re}(s) > 0$, say) and taking $c = \frac32 + \varepsilon$ gives
\begin{align}
\sum_{n > am} &= q^{-\half} \sum_{m, l} \frac{1}{\sqrt{ml}} \frac{1}{2 \pi i} \int_{(\frac32 + \varepsilon)} \left(\frac{X}{q}\right)^s \widetilde{f}(s) (ml)^{-s} ds + O\left(q^{-\half} \left(\frac{X}{q}\right)^{\frac32 + \varepsilon} \frac{a}{q}\right) \\
\label{eq:nbig}
&= q^{-\half} k(q/X) + O\left(q^{-\half} \left(\frac{X}{q}\right)^{\frac32 + \varepsilon} \frac{a}{q}\right),
\end{align}
where recall $k(y)$ is given by \eqref{eq:kdef}.

%Remark. Notice that $k(y)$ has rapid decay for $y \rightarrow \infty$, and for $y$ small satisfies $k(y) = y^{-\half} (b_{-2} \log{y} + b_{-1}) + b_0 + b_1 y +  \dots$.  Furthermore, if $f$ is given by \eqref{eq:fV} then since $G(s) \zeta^2(\half + s)$ has no pole at $s= \half$, then $a_{-2} = a_{-1} = 0$.  

For the terms with $n < am$ we again write $n = am -ql$ with $l > 0$ and $n \equiv -ql \pmod{a}$.  Thus it becomes
\begin{equation}
\sum_{n < am} = \mathop{\sum \sum}_{n \equiv -ql \shortmod{a}} \frac{1}{\sqrt{n\left(\frac{n + ql}{a}\right)}} f\left(\frac{n\left(\frac{n + ql}{a}\right)}{X} \right). %= \sum_{n \equiv -ql \shortmod{a}} f\left(\frac{nlq}{a} \left(1 + \frac{n}{lq}\right)\right).
\end{equation}
Using the Mellin transform, we obtain
\begin{equation}
\label{eq:preliminary}
\sum_{n < am} = \sqrt{\frac{a}{q}} \mathop{\sum \sum}_{n \equiv -ql \shortmod{a}} \frac{1}{\sqrt{nl}} \frac{1}{2 \pi i} \int_{(c)} \left(\frac{aX}{q} \right)^s\widetilde{f}(s) (nl)^{-s} \left(1 + \frac{n}{lq}\right)^{-\half -s} ds.
\end{equation}
Writing $(1 + \frac{n}{lq})^{-\half -s } = 1 + ((1 + \frac{n}{lq})^{-\half -s }-1)$ gives
\begin{equation}
\sum_{n < am} = M.T. + E,
\end{equation}
where
\begin{equation}
M.T. = \sqrt{\frac{a}{q}} \mathop{\sum \sum}_{n \equiv -ql \shortmod{a}} f\left(\frac{nl}{aX/q}\right),
\end{equation}
and
\begin{equation}
E = \sqrt{\frac{a}{q}} \mathop{\sum \sum}_{n \equiv -ql \shortmod{a}} \frac{1}{\sqrt{nl}} \frac{1}{2 \pi i} \int_{(c)} \left(\frac{aX}{q} \right)^s\widetilde{f}(s) (nl)^{-s} \left(\left(1 + \frac{n}{lq}\right)^{-\half -s} -1\right)ds.
\end{equation}
Note
\begin{equation}
\label{eq:MT}
M.T. = \sqrt{\frac{a}{q}} S\left(a,-q;f,\frac{aX}{q} \right).
\end{equation}
Now we bound $E$.  If $nl \geq q^{-1} aX (qX)^{\varepsilon}$ then moving $c$ arbitrarily far to the right shows that these terms contribute a negligible amount to $E$.  For the remaining terms we move $c$ to $\varepsilon$ and use
the approximation \eqref{eq:MVI} again, obtaining
\begin{equation}
E \ll q^{-\frac32} a^{\half} (qX)^{\varepsilon} \mathop{\sum \sum}_{\substack{n \equiv -ql \shortmod{a} \\ nl < \frac{X a}{q} (qX)^{\varepsilon}}} \frac{1}{\sqrt{nl}} \frac{n}{l} + (qX)^{-2007}.
\end{equation}
Now break up the sums over $n$ and $l$ into dyadic segments $N < n \leq 2N$, $L < l \leq 2L$ and apply Lemma \ref{lemma:simple} to obtain
\begin{equation}
 \mathop{\sum \sum}_{\substack{n \equiv -ql \shortmod{a} \\ N < n \leq 2N \\ L < l \leq 2L}} \frac{1}{\sqrt{nl}} \frac{n}{l} \ll \frac{N^{\half}}{L^{\frac32}} \left( \min(L, N) + \frac{LN}{a} \right).
\end{equation}
The worst case is where $L \asymp 1$ and $N \asymp \frac{aX}{q} (qX)^{\varepsilon}$, whence
\begin{equation}
\label{eq:E}
E \ll  q^{-\half} \frac{a}{q} \left(\frac{X}{q}\right)^{\half} \left(1 + \frac{X}{q} \right) (qX)^{\varepsilon}.
\end{equation}
Combining \eqref{eq:diag}, \eqref{eq:nbig}, \eqref{eq:MT} and \eqref{eq:E} completes the proof.
\end{proof}

\section{Asymptotics for $S(q,-a,f)$}
\label{section:asymptoticminus}
In this section we consider the problem of understanding
\begin{equation}
S(q,-a;f,X) = \sum_{n \equiv -am \shortmod{q}} \frac{1}{\sqrt{mn}} f\left(\frac{mn}{X}\right),
\end{equation}
for $0 < a < q$.  The na\"ive idea of replacing $-a$ by $q-a$ and using Theorem \ref{thm:reciprocityAF} does not give a good result for $a < q^{1-\varepsilon}$.  Rather, we apply similar yet slightly different methods as in Section \ref{section:asymptoticplus} to prove the following
\begin{mytheo}
Suppose $0 < a < q$, $(a,q) = 1$, and $f$ is a function satisfying the conditions of Theorem \ref{thm:mainresult}. Then
\begin{equation}
S(q,-a;f,X) = \sqrt{\frac{a}{q}} S(a,q;f, aX/q) + q^{-\half} k(q/X) + O\left(q^{-\half} \frac{a}{q} \left(\frac{X}{q}\right)^{\half} \left(1 + \frac{X}{q} \right) (qX)^{\varepsilon} \right).
\end{equation}
\end{mytheo}
Taking $X=q$ gives \eqref{eq:mainresultminus}.

\begin{proof}
Write $n + am = ql$ and split the sum into two pieces, say $S(q,-a;f,X) = S_1(q,-a;f,X) + S_2(q,-a;f,X)$, where $S_1$ corresponds to the terms with $n < am$ (that is, $n < ql/2$) and $S_2$ corresponds to the terms with $n \geq am$ (that is, $m < ql/(2a))$.  We treat each sum differently.

The sum $S_2(q,-a;f,X)$ is treated similarly as in the proof of Theorem \ref{thm:reciprocityAF} so we shall be brief where the details are similar.  We write $n = ql -am$ and obtain
\begin{align}
S_2(q,-a;f,X) &= \sum_{l > 0} \sum_{m < \frac{ql}{2a}} \frac{1}{\sqrt{m(ql-am)}} f\left(\frac{m(ql -am)}{X}\right) 
\\
&= q^{-\half} \sum_{l > 0} \sum_{m < \frac{ql}{2a}} \frac{1}{\sqrt{ml}} \frac{1}{2 \pi i} \int_{(c)} \left(\frac{X}{q}\right)^s \widetilde{f}(s) (ml)^{-s} \left(1 - \frac{am}{ql}\right)^{-\half-s} ds.
\end{align}
Since $m < \frac{ql}{2a}$, we have $(1 - \frac{am}{ql})^{-\half-s} = 1 + O(|\half + s| |2^s| \frac{am}{ql})$, for $\text{Re}(s) > 0$.  Thus we obtain
\begin{equation}
S_2(q,-a;f,X) = q^{-\half} \sum_{l > 0} \sum_{m < \frac{ql}{2a}} \frac{1}{\sqrt{ml}} \frac{1}{2 \pi i} \int_{(\frac32 + \varepsilon)} \left(\frac{X}{q}\right)^s \widetilde{f}(s) (ml)^{-s} ds + O\left(q^{-\half} \left(\frac{X}{q}\right)^{\frac32 + \varepsilon} \frac{a}{q}\right).
\end{equation}
Extending the summation to all $m > 0$ does not introduce a new error term, so we obtain
\begin{equation}
\label{eq:S2}
S_2(q,-a;f,X) = q^{-\half} k(q/X) + O\left(q^{-\half} \left(\frac{X}{q}\right)^{\frac32 + \varepsilon} \frac{a}{q}\right).
\end{equation}

Now we compute $S_1$.  Following the details of the computation of $\sum_{n < am}$ in the proof of Theorem \ref{thm:reciprocityAF}, we obtain
\begin{equation}
S_1(q,-a;f,X) = M.T. + E,
\end{equation}
where
\begin{equation}
M.T. = \sqrt{\frac{a}{q}} \mathop{\sum \sum}_{\substack{ n < \frac{ql}{2} \\ n \equiv ql \shortmod{a}}} \frac{1}{\sqrt{nl}} \frac{1}{2 \pi i} \int_{(\frac32)} \left(\frac{Xa}{q} \right)^s\widetilde{f}(s) (nl)^{-s} ds
\end{equation}
and $E$ has the same bound as \eqref{eq:E}.  Extending the sum to $n \geq ql/2$ introduces an error of size
\begin{equation}
 \sqrt{\frac{a}{q}} a^{\frac32} \left(\frac{X}{q} \right)^{\frac32} \mathop{\sum \sum}_{\substack{ n \geq \frac{ql}{2} \\ n \equiv ql \shortmod{a}}} \frac{1}{(nl)^2} \ll \sqrt{\frac{a}{q}} a^{\frac32} \left(\frac{X}{q} \right)^{\frac32} \left(\frac{1}{q^2} + \frac{1}{aq} \right) \ll q^{-\half} \frac{a}{q} \left(\frac{X}{q}\right)^{\frac32},
\end{equation}
using Lemma \ref{lemma:simple}.
Thus
\begin{equation}
\label{eq:S1}
S_1(q,-a;f,X) = \sqrt{\frac{a}{q}} S(a,q;f,\frac{aX}{q}) + O\left(q^{-\half} \frac{a}{q} \left(\frac{X}{q}\right)^{\half} \left(1 + \frac{X}{q} \right) (qX)^{\varepsilon}\right).
\end{equation}
Combining \eqref{eq:S2} and \eqref{eq:S1} completes the proof.
\end{proof}

\section{Deduction of Corollary \ref{coro:reciprocity}}
In this section we briefly deduce Corollary \ref{coro:reciprocity} from Theorem \ref{thm:mainresult}.  By Lemma \ref{lemma:summary}, we write
\begin{equation}
M(p,h) = 2 \phi(p) \left( S(p,h;V_{+}) + S(p,-h;V_{-})\right) - 2 \zeta({\textstyle \half})^2(1 - p^{-\half}) + O(p^{-2007}).
\end{equation}
Then by Theorem \ref{thm:mainresult}, we have
\begin{multline}
S(p,h;V_{+}) + S(p,-h;V_{-}) = \sqrt{\frac{h}{p}} \left(S(h,-p;V_{+}) + S(h,p;V_{-})\right) 
\\
+ \frac{1}{\sqrt{h}}(c_1^+\log(p/h) + c_2^+) + \frac{1}{\sqrt{p}} (c_3^+ + c_3^-) + O(p^{-\half + \varepsilon} \frac{h}{p}).
\end{multline}
We compute
\begin{equation}
2 c_1^+ = 1,
\end{equation}
\begin{equation}
2 c_2^+ = {\textstyle \half} (\psi({\textstyle \frac14}) + \psi({\textstyle \frac34})) - \log(\pi) + 2 \gamma = \gamma - \log(8 \pi) = 2\gamma - \log(2 \pi) + \psi({\textstyle \half}):= A,
\end{equation}
and
\begin{equation}
2(c_3^+ + c_3^-) = 2 \frac{1}{2\pi i} \int_{(1)} \pi^{-s} \frac{G(s)}{s}
\frac{\Gamma\left(\frac{\half + s}{2}\right)^2}{\Gamma\left(\frac{1}{4}\right)^2} \zeta({\textstyle \half} + s)^2 ds = \zeta({\textstyle \half})^2,
\end{equation}
the preceding integral having been computed in the proof of Lemma \ref{lemma:summary}.  
Hence we obtain
\begin{multline}
M(p,h) = \frac{p}{\sqrt{h}} (\log(p/h)+A) + \zeta({\textstyle \half})^2 (\sqrt{p} - 2) + 2 \phi(p) \frac{\sqrt{h}}{\sqrt{p}} \left(S(h,-p;V_+) + S(h,p;V_-)\right) 
\\
+ O(p^{-\half + \varepsilon} h).
\end{multline}
Applying Lemma \ref{lemma:summary} again, we obtain
\begin{equation}
2 \phi(p) \frac{\sqrt{h}}{\sqrt{p}}  \left(S(h,-p;V_{+}) + S(h,p;V_{-})\right) = \frac{\phi(p) \sqrt{h}}{\phi(h) \sqrt{p}} \left(M(h,-p) + 2 \zeta({\textstyle \half})^2(1 - h^{-\half}) + O(h^{-C}) \right).
\end{equation}
Putting everything together, we obtain
\begin{multline}
M(p,h) = \frac{p}{\sqrt{h}} (\log(p/h)+A) + \zeta({\textstyle \half})^2 (\sqrt{p} - 2 + 2\frac{ \sqrt{p}}{ \sqrt{h}} \frac{ h}{\phi(h) }(1 - \frac{1}{\sqrt{h}}))
\\
+ \frac{ \sqrt{p}}{ \sqrt{h}} \frac{ h}{\phi(h) }M(h,-p)
+O(p^{-\half + \varepsilon} h + p^{\half} h^{-C}),
\end{multline}
which simplifies slightly to
\begin{multline}
\label{eq:moreprecise}
M(p,h) = \frac{p}{\sqrt{h}} (\log(p/h)+A) + \frac{ \sqrt{p}}{ \sqrt{h}} \frac{ 1}{1-h^{-1}} M(h,-p) 
\\
+\zeta({\textstyle \half})^2 (\sqrt{p} - 2 + 2\frac{ \sqrt{p}}{ \sqrt{h}}  + 2\frac{\sqrt{p}}{h})
+O(p^{-\half + \varepsilon} h + p^{\half} h^{-C}).
\end{multline}
By allowing an error term of size $O(h^{-\half} p^{\half} \log{3h})$, the formula becomes
\begin{equation}
M(p,h) = \frac{p}{\sqrt{h}} (\log(p/h)+A) + \frac{ \sqrt{p}}{ \sqrt{h}} M(h,-p) 
+\zeta({\textstyle \half})^2 \sqrt{p}  + 
+O(p^{-\half + \varepsilon} h + h^{-\half} p^{\half} \log{3h}),
\end{equation}
using the simple bound $|M(h, -p)| \leq M(h,1) \ll h \log{h}$.

\end{document}